\documentclass[11pt]{amsart}
\usepackage[english]{babel}
\usepackage{graphics}
\usepackage{amsfonts}
\usepackage{amsmath}
\usepackage{amssymb}
\usepackage{mathrsfs}
\usepackage{enumerate}
\usepackage{relsize,exscale}
\usepackage{cite}
\usepackage[colorlinks=true,urlcolor=blue,
citecolor=red,linkcolor=blue,linktocpage,pdfpagelabels,
bookmarksnumbered,bookmarksopen]{hyperref}
\usepackage[hyperpageref]{backref}
\usepackage{esint}
\usepackage{epsfig}
\usepackage{amscd}
\usepackage{amsmath}
\usepackage{graphicx}
\usepackage{newlfont}
\usepackage{latexsym}
\usepackage{amsthm}
\usepackage{color}

\newtheorem{theorem}{Theorem}[section]
\newtheorem{lemma}[theorem]{Lemma}

\theoremstyle{remark}

\numberwithin{equation}{section}

\theoremstyle{definition}

\newcommand{\R}{\mathbb{R}}

\newcommand{\OmCh}{\overline{\Omega}}
\newcommand{\OmPu}{\Omega \setminus \{0\}}

\newcommand{\scal}[2]{\langle {#1} , {#2}\rangle}

\newcommand{\Sl}{\Sigma_{\lambda}}
\newcommand{\Smu}{\Sigma_{\mu}}

\newcommand{\SlPu}{\Sigma_{\lambda}\setminus \{0_{\lambda}\}}
\newcommand{\Spl}{\Sigma'_{\lambda}}

\newcommand{\Vla}{V_{\lambda}}

\makeindex
\begin{document}

\title[A symmetry result for polyharmonic problems]{A symmetry result for polyharmonic problems\\
with Navier conditions}\thanks{The authors are members of INdAM/GNAMPA.
The first and the third author are 
partially supported by the INdAM-GNAMPA Project 2018 ``Problemi di curvatura relativi ad operatori ellittico-degeneri''. The second author is supported by
the Australian Research Council Discovery Project 170104880 NEW ``Nonlocal
Equations at Work''. }

\author[S.\ Biagi]{Stefano Biagi}
\author[E.\ Valdinoci]{Enrico Valdinoci}
\author[E.\ Vecchi]{Eugenio Vecchi}

\address[S.\,Biagi]{Dipartimento di Ingegneria Industriale e Scienze Matematiche 
\newline\indent Universit\`a Politecnica della Marche \newline\indent
Via Brecce Bianche, 60131, Ancona}
\email{s.biagi@dipmat.univpm.it}

\address[E.\,Valdinoci]{Department of Mathematics and Statistics
\newline\indent University of Western Australia \newline\indent
35 Stirling Highway, WA 6009 Crawley, Australia}
\email{enrico.valdinoci@uwa.edu.au}

\address[E.\,Vecchi]{Dipartimento di Matematica ``Guido Castelnuovo''\newline\indent
	Sapienza Universit\`a di Roma \newline\indent
	P.le Aldo Moro 5, 00185, Roma, Italy}
\email{vecchi@mat.uniroma1.it}



\date{\today}

\begin{abstract}
We consider an elliptic polyharmonic problem of any order
which takes place in a punctured bounded domain with Navier conditions.
We prove that if the domain is convex in one direction
and symmetric with respect to the reflections
induced by the normal hyperpane to such a direction,
then the solution is necessarily symmetric under this reflection
and monotone in the corresponding direction.
\end{abstract}

\maketitle

%%%_________________________________________%%%%%

%%-----INTRODUCTION-----%%%%%

\section{Introduction}

\noindent Let $\Omega \subset \R^{n}$
(with $n\geq 2$) be a domain (i.e., open, bounded and connected set) 
satisfying the following \emph{structural assumptions}, 
which we assume to be satisfied throughout the paper:
\begin{itemize}
\item $\Omega$ is convex in the $x_1$-direction,
that is, if~$p=(p_1,\dots,p_n)$ and~$q=(q_1,\dots,q_n)$
belong to~$\Omega$ and~$p_j=q_j$ for any~$j\in\{2,\dots,n\}$,
then
$$\text{$(1-t)p+tq$ belongs to~$\Omega$ for all~$t\in[0,1]$};$$
\item $\Omega$ is symmetric
with respect to the hyperplane $\{x_1=0\}$,
that is, 
$$\text{if~$p=(p_1,p_2,\dots,p_n)\in\Omega$ then
$p_0:=(-p_1,p_2,\dots,p_n)\in\Omega$},$$
\item $0$ lies in $\Omega$. 
\end{itemize}
We point out that, since $\Omega$ is connected, the same is true of
$\Omega\setminus\{0\}$.
Let now $m\in{\mathbb{N}}$ be fixed and let $u:\OmPu \to \R$ be a classical solution of
the boundary value problem
\begin{equation}\label{PDE4}
\left\{ \begin{array}{rll}
           (-\Delta)^m u = f(u) & \textrm{in } \OmPu,\\[0.15cm]
	u>0 & \textrm{in } \OmPu,\\[0.15cm]
	u=- \Delta u = \ldots = (-\Delta)^{m-1}u= 0 & \textrm{on } \partial \Omega,\\[0.15cm]
			\displaystyle\inf_{\OmPu} (-\Delta)^j u > -\infty & j=1,\ldots,m-1.
					\end{array}\right.
\end{equation}
By classical solution we mean $u \in C^{2m}(\OmPu) \cap C^{2m-2}(\OmCh \setminus \{0\})$.
We notice that, if one aims to prove the \emph{existence}
of such a solution, some regularity on the boundary $\partial\Omega$ of $\Omega$ must be required
see e.g., \cite[Theorem 2.19]{GGS}.
On the other hand, we do not need to take any additional
assumption here, since we are assuming a priori that a solution
exists and we aim at proving its symmetry and monotonicity
properties.
\medskip

The main aim of this paper is to study symmetry properties of the 
positive solutions of~\eqref{PDE4}.
To this end we first observe that, if $u$ is such a solution and if we set
$$u_1 := u, \qquad u_{i+1} := (-\Delta)^i u \quad (\text{for $i = 1,\ldots,m-1$}),$$
then then vector-valued function $U = (u_1,\ldots,u_m)$ belongs
to $C^2(\Omega\setminus\{0\})\cap C(\overline{\Omega}\setminus\{0\})$
and satisfies the 
following second-order elliptic system
\begin{equation*}
\left\{ \begin{array}{rll}
     -\Delta u_i = u_{i+1} & \textrm{in } \OmPu, & i=1,\ldots,m-1,\\[0.12cm]
		 -\Delta u_m = f(u_1) & \textrm{in } \OmPu, & i=m,\\[0.12cm]
		 u_i = 0 & \textrm{on } \partial \Omega, & i=1,\ldots,m,\\[0.12cm]
		 u_1>0 & \textrm{in } \OmPu\\[0.12cm]
		\displaystyle\inf_{ \OmPu} u_i > -\infty,  & i= 2, \ldots, m. 
		  \end{array}\right.
\end{equation*}
In the present paper we assume that
\begin{equation}\label{f1} 
\tag{f1}
f \in \mathrm{Lip(\R^+)}, \quad f(0) \geq 0 \quad \textrm{and } \quad f \textrm{ is non-decreasing}.
\end{equation}
We note that condition~\eqref{f1}, together with the weak maximum principle in punctured domains 
(see~\cite[Lemma 2.1]{CLN2}), the lower-boundedness of the $u_i$'s
and the standard strong maximum principle, yields 
the positivity of the components $u_1,\ldots,u_m$ of $U$
in $\OmPu$. \medskip

We are now ready to state one of our main results.
\begin{theorem} \label{thm.main.pre}
Let $\Omega\subseteq\R^n$ be a domain satisfying the structural assumptions introduced
  above and let $f:\R^+\to\R$ satisfy~\eqref{f1}. Moreover,
  let $u\in C^{2m}(\Omega\setminus\{0\})\cap C^{2m-2}(\overline{\Omega}\setminus\{0\})$
  be a classical solution of the $2m$-th order boundary value problem~\eqref{PDE4}.
  
Then the following facts hold true:
  \begin{itemize}
   \item[{(1)}] $u$ is symmetric in $x_1$, i.e., $u(x_1,x_2,\ldots,x_n) = u(-x_1,x_2,\ldots,x_n)$
   for every $x\in\Omega$;\vspace*{0.15cm}
   \item[{(2)}]  $u$ is decreasing with respect to $x_1$ on $\Omega\cap\{x_1 > 0\}$.
  \end{itemize}
 \end{theorem}
 
\noindent As a byproduct of this result, we have that if~$\Omega$ is a ball,
then the solution~$u$ is necessarily radial and radially decreasing.
On the one hand,
when~$m=1$, our result recovers the classical result in~\cite{Serrin}.
On the other hand, when~$m=2$, problem~\eqref{PDE4}
finds na\-tu\-ral applications in engineering, for instance
in the description
of ``hinged'' rigid plates, see e.g.~\cite{GGS}.
Other examples of polyharmonic operators
naturally
appear in the phase separation of a two component system,
as described by the Cahn-Hilliard
equation (see~\cite{CH}),
and when comparing the pointwise values of a function with
its average (see~\cite{PIZ}). 

As is well-known, the literature concerning symmetry results for
elliptic PDEs is extremely wide, and is far from our scopes
to present here an exhaustive list of references.
We must mention the seminal papers 
\cite{Serrin, GNN, BN} for the use of the moving planes method
in the PDEs setting; we also highlight \cite{Terracini96, CLN2, Sciunzi17, MPS17, EFS18}
(for symmetry results for singular solutions of scalar semilinear equations
in local and non-local setting)
and \cite{Troy81, deFigueiredo, FGW, BGW08, DP, CoVe, CoVe2}
(for symmetry results for semilinear polyharmonic problems and cooperative
elliptic systems).

\medskip

In our framework, we will deduce Theorem~\ref{thm.main.pre}
as a particular case of a more general result, which is valid for
Pizzetti-type superpositions of polyharmonic operators, with
suitable structural assumptions on the coefficients. To this end,
given~$\alpha_1,\dots,\alpha_m\in\R$, we define
~$\alpha:=
(\alpha_1,\dots,\alpha_m)\in\R^m$, and we consider the
characteristic polynomial expansion
\begin{equation}\label{8A}
\det\left( \begin{matrix}\alpha_1+t & 0 &\dots &0\\
0& \alpha_2+t&\dots&0\\
\vdots& \vdots& \ddots&\vdots\\
0&0&\dots&\alpha_m+t
\end{matrix}
\right)=\sum_{k=0}^m s_k(\alpha)\,t^k,\end{equation}
where $s_m(\alpha) = 1$ (independently on $\alpha$) and, for every $k = 0,\ldots,m-1$, we have
\begin{equation}\label{8B}
s_k(\alpha) = \sum_{1\leq i_1 < \ldots < i_{m-k}\leq m}\alpha_{i_1}\cdots\alpha_{i_{m-k}}.\end{equation}
We stress that, by the
\emph{Descartes rule of signs}, 
the positivity of $s_0(\alpha),\ldots,s_{m-1}(\alpha)$ is equivalent
to the positivity of all $\alpha_i$; see also Lemma \ref{8}
for a self-contained proof. \medskip

\noindent Then, we consider the equation
\begin{equation}\label{PDE444}
\left\{ \begin{array}{rll}
    \sum_{k=0}^m s_k(\alpha)(-\Delta)^k u = f(u) & \textrm{in } \OmPu,\\[0.15cm]
	u>0 & \textrm{in } \OmPu,\\[0.15cm]
	u=- \Delta u = \ldots = (-\Delta)^{m-1}u= 0 & \textrm{on } \partial \Omega,\\[0.15cm]
			\displaystyle\inf_{\OmPu} (-\Delta)^j u > -\infty & j=1,\ldots,m-1.
					\end{array}\right.
\end{equation}
The symmetry and monotonicity result in this general setting goes as follows.

\begin{theorem} \label{thm.main}
Assume that~$s_0(\alpha),\dots,s_{k-1}(\alpha)\in[0,+\infty)$.
Let $\Omega\subseteq\R^n$ be a domain satisfying the
structural assumptions introduced
  above and let $f:\R^+\to\R$ satisfy~\eqref{f1}. Moreover,
  let $u\in C^{2m}(\Omega\setminus\{0\})\cap
  C^{2m-2}(\overline{\Omega}\setminus\{0\})$
  be a classical solution of the $2m$-th
  order boundary value problem~\eqref{PDE444}.
  
Then the following facts hold true:
  \begin{itemize}
   \item[{(1)}] $u$ is symmetric in $x_1$, i.e., $u(x_1,x_2,\ldots,x_n) = u(-x_1,x_2,\ldots,x_n)$
   for every $x\in\Omega$;\vspace*{0.15cm}
   \item[{(2)}]  $u$ is decreasing with respect to $x_1$ on $\Omega\cap\{x_1 > 0\}$.
  \end{itemize}
\end{theorem}
\noindent Notice that Theorem~\ref{thm.main.pre} is a straightforward consequence of
the above Theorem~\ref{thm.main}, by choosing~$\alpha_1=\dots=\alpha_m=0$
(which gives that~$s_0(\alpha)=\dots=s_{k-1}(\alpha)=0$). \medskip

We now spend few words on the regularity assumption of $
f$ in \eqref{f1}. When $m = 1$, in \cite{CLN2} 
the analogue of Theorem \ref{thm.main.pre} 
is proved
under the weaker assumption that
$f$
is only \emph{locally} Lipschitz-continuous
(and possibly depending on the spatial variable $x$).
In our case, we instead
assume a \emph{global} Lipschitz assumption in~\eqref{f1},
since boundedness issues become more involved
when $m\geq 2$ (roughly speaking, for the case of systems, the positivity
of one component in a subdomain does not imply the positivity
of the other components). For a similar reason, we also
assume the bound
\begin{equation}\label{TR}
\inf_{\OmPu} u_i >-\infty.\end{equation} Indeed, when $m=1$
dealing with positive solutions implies immediately the former bound.
On the other hand, for $m \geq 2$, while this is still true for $u_1$,
this is not automatically inherited by the other components of the system.
We stress that an assumption similar to~\eqref{TR}
has been made also by Troy in \cite{Troy81},
who asked for the positivity of all the components of the solution of a semilinear elliptic system. 
We think that it is an interesting {\em open problem}
to further investigate
whether either the global Lipschitz
regularity assumption or the bound from
below of the $(-\Delta)^j u$ 
can be relaxed as in \cite{CLN2}.
\medskip

The rest of the paper is organized as follows. After stating some
notation, in Section~\ref{Sec2} we present the main technical results,
related to suitable versions of the maximum principle for cooperative
systems,
to which we can in turn reduce the setting of problem~\eqref{PDE4}.
Then, Theorem~\ref{thm.main}
will be proved in Section~\ref{Sec3} obtaining the symmetry result
by the moving plane and reflection methods (which need to be
suitably
adapted to take into account the cases of higher order operators
and not fully coupled cooperative systems),
and the monotonicity result by the Hopf's Lemma.

%%______________________________________%%%

%%-----------------Sezione 2-----------------------%%%

\section{Notation, assumptions and preliminary results}\label{Sec2}
We introduce some notation and the standing assumptions used along the paper.
For a function $U:\Omega\to \R^{m}$, $U=(u_1,\ldots,u_m)$, we say that $U \geq 0$
if $u_i \geq 0$ for every $i=1,\ldots,m$.
The notation for the moving plane technique is
as in the paper of Serrin \cite{Serrin}, and it goes as follows. 
Given a point $x\in \R^n$, we denote by $(x_1, \dots,x_n)$ its components,
and, when more practical, we equivalently write $x = (x_1,x')\in\R\times\R^{n-1}$.
For a given unit vector $e \in \R^n$ and for $\lambda\in\R$, we define the hyperplane
$$
T_{\lambda}:= \{ x \in \R^n: \scal{e}{x}=\lambda \}.
$$
From now on, unless otherwise stated, without loss of generality, we assume 
that $e=e_1$, i.e. the normal to $T_{\lambda}$ is parallel
to the $x_1$-direction. \\
To simplify the readability, we further assume that
\begin{equation}\label{A1}
\tag{A1}
\sup_{\Omega} x_1 = 1. 
\end{equation}
Now, for $\lambda \in (0,1)$, we define
\begin{equation}\label{x_lambda}
x_{\lambda}:= (2\lambda - x_1, x').
\end{equation}
We stress that~\eqref{x_lambda} may lead to points that
{\it do not belong} to $\OmCh$, e.g. $0_{\lambda}= (2\lambda, 0, \ldots,0)\not\in\OmCh$
for $\lambda>\tfrac{1}{2}$, in view of~\eqref{A1}.
 Proceeding further with the notation,
given any $\lambda\in\mathbb R$, we introduce the possibly empty set 
\begin{equation*}
\Sl:=\{x\in\Omega\, : \,x_1>\lambda\}
\end{equation*}
and its reflection with respect to
$T_{\lambda}$, 
$$
\Spl:=\{x_{\lambda} \in \mathbb R^n \,:\,x\in\Sl\}.
$$

\noindent Since $\Omega \subset \R^n$ is bounded, by 
\eqref{A1} we have that
$T_{\lambda}$ does not touch $\overline{\Omega}$ for $\lambda > 1$;
moreover, 
$T_{1}$ touches $\overline{\Omega}$
and, for every $\lambda \in (0,1)$, the hyperplane $T_{\lambda}$
cuts off from $\Omega$ the portion $\Sl$. 
At the beginning of this process, the reflection $\Spl$ of $\Sl$ will be contained in $\Omega$.
%Now, we define the value $\lambda_1$ as follows
%\begin{equation}\label{DefLambda1}
%\lambda_1:=\sup\{\lambda<\lambda_0\,:\,\mbox{(i) or (ii) is verified}\},
%\end{equation}
%where
%\begin{itemize}
%\item[(i)] $\Spl$ is internally tangent to the boundary $\partial \Omega$ at a certain point $P \notin T_{\lambda}$;
%\item[(ii)] $T_{\lambda}$ is orthogonal to the boundary $\partial \Omega$ at a certain point $Q \in T_{\lambda} \cap \partial \Omega$.
%\end{itemize}
%By construction, 
%$$
%\Sigma'_{\lambda} \subset \Omega\quad\mbox{for every }\lambda\in[\lambda_1,1).
%$$ 
%Nevertheless, by further decreasing the value of $\lambda$ below $\lambda_1$,
%$\Sigma'_{\lambda}$ might still be contained in $\Omega$. Therefore we define
%the value  
%\begin{equation}\label{DefLambda2}
%\lambda_2:=\inf\{\lambda<\lambda_0\,:\, \Sigma'_{\lambda} \subset \Omega \}.
%\end{equation}

Let now $u\in C^{2m}(\Omega\setminus\{0\})\cap C^{2m-2}(\overline{\Omega}\setminus\{0\})$
be a classical solution of the $2m$-th order boundary value problem
\eqref{PDE444}. If $\alpha = (\alpha_1,\ldots,\alpha_m)\in \R^m$, we set
\begin{equation} \label{eq.defuivere}
 u_1 := u, \qquad u_{i+1} = (-\Delta+\alpha_i)u_i \quad\big(\text{for $i = 1,\ldots,m-1$}\big).
\end{equation}
Moreover, for every 
$i=1,\ldots,m$ and for every fixed $\lambda\in (0,1)$, we introduce the functions
$u_{i}^{(\lambda)}, v_{i}^{(\lambda)}: \SlPu \to \R$ defined by
$$u_{i}^{(\lambda)}(x):= u_{i}(x_{\lambda}) \qquad \text{and}
\qquad v_{i}^{(\lambda)}(x):= u_{i}^{(\lambda)}(x)-u_{i}(x).$$
Finally, to simplify the notation, we also define
\begin{equation}\label{def:U}
U_{\lambda}(x) := \left( u_{1}^{(\lambda)}, u_{2}^{(\lambda)}, \ldots, u_{m}^{(\lambda)} \right) \quad x \textrm{ in } \SlPu,
\end{equation}
\noindent and
\begin{equation}\label{def:V}
V_{\lambda}(x) := \left( v_{1}^{(\lambda)}, v_{2}^{(\lambda)}, \ldots, v_{m}^{(\lambda)} \right) \quad x \textrm{ in } \SlPu,
\end{equation}
We observe that, since $\Omega$ is symmetric and convex with respect to $T_0=\{x_1=0\}$,
the above definitions~\eqref{def:U} and~\eqref{def:V} 
are well-posed for every $\lambda\in[0,1)$.

\medskip
We now recall the maximum principle in small domains for cooperative systems proved
by de Figueiredo in \cite{deFigueiredo}.

\begin{lemma}[Proposition 1.1 of \cite{deFigueiredo}]\label{MaxSmall}
Let $\Omega\subset\mathbb R^n$ be a bounded domain,
$$
\mathcal L_m:=\left(
\begin{array}{ccccc}
\Delta & 0 & 0 & \ldots & 0\\
0 & \Delta & 0 & \ldots & 0\\
0 & 0 & \ddots & 0 & 0\\
0 & 0 & \ldots & \Delta & 0\\
0 & 0 & \ldots & 0 & \Delta
\end{array}
\right)\quad \mbox{and}\quad
A(x):= (a_{ij}(x))_{i,j=1,\ldots,m},
$$
with $a_{ij}\in L^\infty(\Omega)$ for $i,j=1,\ldots,m$ and $a_{ij}\ge 0 $ for $i\neq j$. 
Suppose that $U:=(u_1,\ldots,u_m)$ is a vector-valued function in $C^2(\Omega)$ such that
$$
\begin{cases}
-(\mathcal L_m+A {(x)})\,U\ge0 & \mathrm{in\,\Omega},\\
\displaystyle{\liminf_{x\to y}}\,U(x)\ge0 & \mathrm{for\,every}\,y\in\partial\Omega.
\end{cases}
$$
Then, there exists $\delta=\delta(n,\mathrm{diam}(\Omega))>0$ such that if $|\Omega|<\delta$, $U\ge 0$ in $\Omega$ 
\end{lemma}
The following technical result can be seen as a slight
variation of Lemma~\ref{MaxSmall}
and as an extension of \cite[Proposition 2.1]{CLN2} to the case of (special) cooperative systems.
\begin{lemma}\label{PerG}
Let $\xi\in\R^n, r>0$ be fixed and let
\begin{equation}\label{TSc1}
\Omega \subset B_{r}(\xi)\end{equation} be an open set. Moreover, let $x_0\in\Omega$
be arbitrarily chosen,
let $\mathcal{L}_m,\,A(x)$ be as in Lemma~\ref{MaxSmall} and let
$U = (u_1,\ldots,u_m)\in C^{2}(\Omega \setminus \{x_0\})$ be such that
\begin{equation}\label{TSc2}\begin{cases}
 -(\mathcal{L}_m+A(x))\,U \geq 0 & \mathrm{in\,}\Omega \setminus \{x_0\}, \\[0.1cm]
  \liminf\limits_{x\to y}U(x) \geq 0 & \mathrm{for\,every}\,y\in\partial \Omega, \\
 \inf_{\Omega\setminus\{x_0\}} u_i > -\infty & i = 1,\ldots, m.
 \end{cases}\end{equation}
Then, for sufficiently small~$r >0$, we have
$$U \geq 0, \quad \mathrm{in}\,\Omega \setminus \{x_0\}.$$
\end{lemma}
\begin{proof}
Without loss of generality, we can assume that $\xi = x_0 = 0$.
We then consider the vector-valued function $W := U + \epsilon H$ given by
\begin{equation*}
W:= \left( \begin{array}{c}
               u_1 + \epsilon h\\
							 u_2 + \epsilon h\\
							 \vdots \\
							 u_m + \epsilon h \\
					\end{array}\right),
\end{equation*}
\noindent where $h(x):= (- \ln(|x|))^{a}$, with $a \in (0,1)$ to be chosen,
as in \cite{CLN2}.
Obviously, we have that
\begin{equation}\label{TSc3}
{\mbox{$h \geq 0$ on $B_r(0)$, and $h \to +\infty$ as $|x|\to 0$.}}\end{equation}
Moreover, we claim that
$$\mathcal{L}_m h \leq 0, \quad \textrm{in } B_{r}(0)\setminus \{0\},$$
if $r>0$ is small enough. Indeed, since $a_{ij}\in L^\infty(\Omega)$ for any $i,j$,
we can define
$$C_{ij}:= \sup\{a_{ij}(x): x \in \Omega)\}\qquad{\mbox{and}}\qquad
K:= \max_{i = 1,\ldots,m}\sum_{j = 1}^m C_{ij}.$$ 
Then, a direct computation shows that
\begin{equation*}
 \begin{split}
\Delta h + K h & = - \frac{a(-\ln(|x|))^{a-1}}{|x|^2} (n-2) + \frac{a(a-1)}{|x|^2} (\ln(|x|))^{a-2} + 
K (-\ln(|x|))^a \\[0.35cm]
& = \begin{cases}  
 \displaystyle - \frac{a(-\ln(|x|))^{a-1}}{|x|^2}\cdot\Big(n - 2 + o(|x|)\Big) \\
 \\
 \displaystyle \frac{(\ln(|x|))^{a-2}}{|x|^2} \cdot\Big(
a(a-1) + o(|x|)\Big)
\end{cases} \qquad \big(\text{as $|x|\to 0$}\big).
\end{split}
\end{equation*}
Since $n\geq 2$ and $a\in (0,1)$, if $r \ll 1$ is sufficiently small we obtain
\begin{equation} \label{eq.Lmhleqzero}
 \Delta h + K h \leq 0 \quad \text{in $B_r(0)\setminus\{0\}$}.
\end{equation}
Now, by definition of $K$ and the fact that
$h\geq 0$ on $\R^n\setminus\{0\}$, from~\eqref{TSc1}, \eqref{TSc2}
and~\eqref{eq.Lmhleqzero} we get 
that
\begin{equation} \label{eq.tousemin1}
 \begin{split}
\Big[(\mathcal{L}_m + A(x))\,W\Big]_i & = \Big[(\mathcal{L}_m + A(x))\,U\Big]_i
+ \epsilon\,\Big(\Delta h + \big(\textstyle\sum_{j = 1}^ma_{ij}(x)\big)\cdot h\Big) \\[0.2cm]
& \leq \epsilon \Big(\Delta h + K h\Big) \leq 0 \qquad 
\text{for $i = 1,\ldots,m$}.
\end{split}
\end{equation}
Furthermore, recalling~\eqref{TSc3},
we have 
that
\begin{flalign}
& \bullet\,\,
\text{$\liminf\limits_{x\to y}W(x) \geq \liminf\limits_{x\to y}U(x) \geq 0$ for every
$y\in\partial\Omega$};
\label{eq.tousemin2} \\
& \bullet\,\,
\text{$W_i(x) = u_i(x)+\epsilon\,h(x)\to+\infty$ as $|x|\to 0$ 
(since $\inf_{\Omega \setminus\{x_0\}} u_i > -\infty$)}. \label{eq.tousemin3} &&  
\end{flalign}
 Gathering together~\eqref{eq.tousemin1}, \eqref{eq.tousemin2} and~\eqref{eq.tousemin3}
(and by possibly shrinking $r$), we are entitled to apply the weak minimum principle
in Lemma~\ref{MaxSmall}, obtaining $W = U+\epsilon H\geq 0$ in $\Omega\setminus\{0\}$.
From this, taking the limit as $\epsilon \to 0^{+}$,
we conclude that~$
U \geq 0$ in~$\Omega \setminus \{0\}$.
\end{proof}
For the sake of completeness, and due to its relevance in the sequel,
we close this section with the following
 more general version of 
\cite[Lemma 2.1]{CLN2}.
\begin{lemma} \label{lem.WMPbuco}
 Let $\Omega\subseteq\R^n$ be an open and bounded set and let $x_0\in\Omega$ be fixed.
 Moreover, let $a(x)$ be a non-negative function on $\Omega\setminus\{x_0\}$.
 Finally, let $w\in C^2(\Omega\setminus\{x_0\})$ be s.t.
 \begin{equation} \label{eq.systemw}
  \begin{cases}
  (-\Delta + a(x))\,w(x) \geq 0 & \mathrm{in}\,\Omega\setminus\{x_0\}, \\[0.15cm]
  \displaystyle \liminf_{x\to y}w(x) \geq 0 &  
  \mathrm{for\,every}\,y\in\partial\Omega,\\[0.15cm]
	\inf w > -\infty & \mathrm{in }\, \Omega \setminus \{x_0\}.
  \end{cases}
 \end{equation}
 Then, $w\geq 0$ in $\Omega\setminus\{x_0\}$.
\end{lemma}
 \begin{proof}
  For every $\epsilon > 0$, we consider the function
  $$w_\epsilon(x) := w(x) + \epsilon\,G_\Omega(x_0,x) 
  \qquad \text{$x\in\Omega\setminus\{x_0\}$},$$
  where $G_\Omega(x_0,\cdot)$ is the Green function of $\Delta$
  related to the open set
  $\Omega$ and with pole at $x_0$ (note that, 
  as $\Omega$ is bounded, such a function always
  exists, even if $n = 2$). More precisely,
  $$G_\Omega(x_0,x) = \Gamma(x_0,x) - h_{x_0}(x) \qquad
  \text{$x_0\in\Omega$ and $x\in\Omega\setminus\{x_0\}$},$$
  where $\Gamma(x_0,\cdot)$ is the global fundamental solution
  of $\Delta$ (with pole at $x_0$) and $h_{x_0}$
  is the greatest harmonic minorant of $\Gamma(x_0,\cdot)$
  on $\Omega$, i.e. (see \cite{AG}),
  $$
   h_{x_0} = \sup\big\{\text{$h$\,:\,$h$ is harmonic on $\Omega$ and
   $h\leq \Gamma(x_0,\cdot)$ on $\Omega$}\big\}.
  $$
  
  Since $G_\Omega(x_0,\cdot)$ is a non-negative harmonic function out of $\{x_0\}$
  and, by assumption, $a\geq 0$ on $\Omega\setminus\{x_0\}$, it readily follows
  from~\eqref{eq.systemw} that
  $$(-\Delta + a(x))\,w_\epsilon(x) = (-\Delta+a(x))\,w(x) + \epsilon\,a(x)\,G_\Omega(x_0,x)
  \geq 0, \quad \text{for $x\in\Omega\setminus\{x_0\}$}.$$
  Moreover, again by~\eqref{eq.systemw},
  for every $y\in\partial\Omega$ we have that
  $$\liminf_{x\to y}w_\epsilon(x) \geq \liminf_{x\to y}w(x) \geq 0.$$
  Finally, since $G_\Omega(x_0,x)\to +\infty$ as $x\to x_0$ and
  $\inf_{\OmPu}w >-\infty$, 
	we also infer that
  $$\liminf_{x\to x_0}w_\epsilon(x) = +\infty.$$
  By combining all these facts, we are entitled to apply the classical weak maximum principle
  for $-\Delta+a(x)$ on $\Omega\setminus\{x_0\}$, which gives
  $w_\epsilon\geq 0$ on $\Omega\setminus\{x_0\}$.   
  {F}rom this, by sending
  $\epsilon\to 0^+$, we conclude that $w\geq 0$
  throughout $\Omega\setminus\{x_0\}$, as desired.
 \end{proof}
\medskip

%%______________________________________%%%

%-----------------------Section 3----------------------%%

\section{Proof of Theorem~\ref{thm.main}}\label{Sec3}
In the present setting, we can now perform the proof
of Theorem~\ref{thm.main}. For simplicity we separate the
proof of the first claim in Theorem~\ref{thm.main},
which is the core of the argument, from the proof
of the second claim, which is mostly straightforward.

\begin{proof}[Proof of Theorem~\ref{thm.main} - (1)]
% We start with some preliminary observations towards
% the proof of Theorem~\ref{thm.main}.
First of all we observe that, if
$u_1,\ldots,u_m$ are as in \eqref{eq.defuivere},
%the $2m$-th order boundary vale problem
%\eqref{PDE444} is equivalent to the
they satisfy the following system
\begin{equation} \label{eq.systemperu}
\left\{ \begin{array}{rll}
     (-\Delta +\alpha_i) u_i = u_{i+1} & \textrm{in } \OmPu, & i=1,\ldots,m-1,\\[0.12cm]
		 (-\Delta + \alpha_m) u_m = f(u_1) & \textrm{in } \OmPu,\\[0.12cm]
		 u_1>0 & \textrm{in } \OmPu.
		  \end{array}\right.
\end{equation}

\noindent Analogously, for every fixed $\lambda\in (0,1)$,
the functions $v_{i}^{(\lambda)}$ satisfy
in the open set $\SlPu$ the following (system of) PDEs
\begin{equation} \label{eq.systempervlambda}
 \begin{split}
(-\Delta +\alpha_i) v_{i}^{(\lambda)} & = (-\Delta +\alpha_i) u_i^{(\lambda)} + 
(-\Delta +\alpha_i) u_i \\[0.2cm]
& = \left\{ \begin{array}{rll}
          u_{i+1}^{(\lambda)} - u_{i+1} = v_{i+1}^{(\lambda)}, & i=1,\ldots,m-1,\\[0.15cm]
               f(u_1^{(\lambda)}) - f(u_1) = c(x,\lambda) v_{1}^{(\lambda)}, & i=m,
           \end{array}\right.
 \end{split}
\end{equation}
\noindent where $c(x,\lambda) \in [0,M]$ for a certain positive constant $M>0$, due to~\eqref{f1}.
Now, by 
the very definition of $u_1,\ldots,u_m$, it is not difficult to see that, 
for $j = 1,\ldots, m$, the function
$u_j$ is a linear combination 
with non-negative coefficients
of $(-\Delta)^i u$ for $0 \leq i \leq j$.
As a consequence,
due to the last two conditions in~\eqref{PDE444}, we have that
\begin{equation} \label{eq.UvanishesdeOmega}
 \text{$U = (u_1,\ldots,u_m) \equiv 0$ on $\partial \Omega$}
 \quad \text{and} \quad
 \text{$\inf_{\OmPu} u_i >-\infty$ for every $i = 2,\ldots,m$}.
\end{equation} 
{F}rom this, by Lemma
\ref{lem.WMPbuco} and the classical strong maximum principle
(both applied to each scalar equation in
system \eqref{eq.systemperu}), we conclude that
\begin{equation} \label{eq.Ustrictlypositive}
U = (u_1,\ldots,u_m) > 0 \qquad \text{in $\Omega\setminus\{0\}$}.
\end{equation}
We then observe that, since $\lambda$ is \emph{strictly positive}, 
the reflection of $\partial\Sigma_\lambda\cap\partial\Omega$
with respect to the hyperplane $T_\lambda$ is entirely contained
in $\Omega$ (remind the structural assumptions satisfied by $\Omega$).
As a consequence, by~\eqref{eq.UvanishesdeOmega} and
\eqref{eq.Ustrictlypositive}, we derive that
\begin{equation} \label{eq.Vlambdasulbordogenerale}
 \text{$V_\lambda \geq 0$ on $\partial\SlPu$\,\,but\,\,$V_\lambda\not\equiv 0$ on the same set}.
\end{equation}
We explicitly point out that, if $\lambda\neq 1/2$, 
we have that~$0_\lambda\notin\partial\Sigma_\lambda$. \medskip

\noindent Gathering all these facts, we are thus dealing with the following system:						
\begin{equation}\label{SysUV}
\left\{ \begin{array}{rll}
            (-\Delta+\alpha_i) v_{i}^{(\lambda)} = v_{i+1}^{(\lambda)}, 
            &\textrm{in } \SlPu, & i=1,\ldots,m-1,\\[0.15cm]
            (-\Delta+\alpha_m) v_{m}^{(\lambda)} = c(x,\lambda) v_{1}^{(\lambda)}, & \textrm{in } \SlPu,\\[0.15cm]
			v_{i}^{(\lambda)} \geq 0, v_{i}^{(\lambda)} \not\equiv 0, 
			& \textrm{on } \partial \SlPu, &i=1,\ldots,m.
				\end{array}\right.			
\end{equation}
Furthermore, since
$u_1,\ldots,u_m$ are non-negative and continuous
on $\overline{\Omega}\setminus\{0\}$ and,
for every fixed $\lambda \in (0,1)$, the set
$\overline{\Sigma_\lambda}$ 
is compactly contained in $\overline{\Omega}\setminus\{0\}$, we have
\begin{equation} \label{eq.vilambdabd}
 \inf_{\SlPu}v^{(\lambda)}_i 
\geq \inf_{\SlPu}(-u_i) > -\infty, \quad
\text{for every $i = 1,\ldots,m$}.
\end{equation}
It is immediate to check that the system in~\eqref{SysUV} is (weakly) cooperative and satisfies 
the assumptions of Lemma~\ref{MaxSmall}.
This implies that, for $\lambda$ very close to 1,
$\Vla \geq 0$ in $\Sl$
(note that, if $\lambda \sim 1$, then $0_\lambda\notin\OmCh$).
Moreover, by the strong maximum principle
applied to the each scalar equation of \eqref{SysUV}, we have
\begin{equation} \label{1Ineq}
\Vla > 0 \quad \textrm{in } \Sl.
\end{equation}
We can then define 
$$
\mathcal{I} := \left\{\lambda\in(0,1)\,:\, V_t > 0 \mbox{ in } \Sigma_t\setminus\{0_t\}\,\,
\forall\,t\in (\lambda,1)\right\} \quad
\text{and} \quad \mu := \inf\mathcal{I}.
$$
We list below some useful properties of $\mu$. \medskip
 \begin{itemize}
  \item[(i)] $\mu\in [0,1)$. 
  This is a straightforward consequence of~\eqref{1Ineq}. \medskip
  \item[(ii)] $V_t > 0$ on $\Sigma_t\setminus\{0_t\}$ for every $t \in (\mu,1)$. Indeed,
  let $t\in (\mu,1)$ be fixed. Since $\mu = \inf\mathcal{I}$, it is possible to find
 $\lambda\in\mathcal{I}$ such that
 $\mu<\lambda< t$.
 As a consequence, since $\lambda\in\mathcal{I}$, we conclude that
 $V_t > 0$ on $\Sigma_t\setminus\{0_t\}$, as claimed. \medskip
 
 \item[(iii)] $V_\mu \geq 0$ on $\Sigma_\mu\setminus\{0_\mu\}$. Indeed,
 if $\xi \in \Sigma_\mu\setminus\{0_\mu\}$ is fixed, there exists
 a small $\epsilon_0 > 0$ such that
 $\xi\in\Sigma_{\mu+\epsilon}\setminus\{0_{\mu+\epsilon}\}$
 for every $\epsilon\in [0,\epsilon_0)$. Thus,
 by (ii) we have
 $$V_{\mu+\epsilon}(\xi) > 0 \quad \text{for every $\epsilon\in [0,\epsilon_0)$}.$$
 Taking the limit as $\epsilon\to 0^+$, we conclude that
 $V_\mu(\xi) \geq 0$, as claimed. \medskip
 
 \item[(iv)] If $\mu > 0$, then $V_\mu > 0$ on $\Sigma_\mu\setminus\{0_\mu\}$. Indeed,
 by (iii) we know that $V_\mu\geq 0$ on its domain. 
 Moreover, 
 since $V_\lambda$ solves~\eqref{SysUV} and~$c(x;\lambda)\geq 0$, we have
 $$\text{$(-\Delta +\alpha_i) v^{(\mu)}_i \geq 0$ on $\Sigma_\mu\setminus\{0_\mu\}$}
 \qquad (\text{for every $i = 1,\ldots,m$}.$$
 On the other hand, if $\mu > 0$, using again~\eqref{SysUV}, we see that
 $$\text{$v^{(\mu)}_i\not\equiv 0$ on 
 $\partial\Sigma_\mu\setminus\{0_\mu\}$} \qquad \text{for $i = 1,\ldots,m$}.$$
 Thus, by the strong maximum principle for $(-\Delta +\alpha_i)$
 (and the fact that $\Sigma_\mu\setminus\{0_\mu\}$ is connected), we conclude
 that
 $V_\mu > 0$ on $\Sigma_\mu\setminus\{0_\mu\}$, as claimed.
 \end{itemize}
\medskip
\noindent The goal is to show that 
\begin{equation}\label{DM}\mu = 0.
\end{equation} 
Indeed, if this is the case, by the above (iii) we have
(for $x\in\Sigma_0 = \Omega\cap\{x_1 > 0\}$)
$$u(-x_1,x_2,\ldots,x_n) = \Big[U_0(x)\Big]_1 \geq \Big[U(x)\Big]_1 = u(x_1,\ldots,x_n).$$
Thus, by applying this result to the function 
$\Omega\ni x\mapsto w(x) := u(-x_1,\ldots,x_n)$
(which solves the same PDE in~\eqref{PDE4}), we conclude that,
for every $x\in\Omega\cap\{x_1 > 0\}$,
\begin{equation}\label{MM}u(x_1,\ldots,x_n) = w(-x_1,\ldots,x_n)
\geq w(x_1,\ldots,x_n) = u(-x_1,\ldots,x_n),\end{equation}
and this proves that $u$ is symmetric in $x_1$, as desired.\bigskip

\noindent As in \cite{CLN2}, we prove \eqref{DM}
by contradiction considering three possible cases:
$$
\mu \in \left( \tfrac{1}{2},1 \right), \qquad
\mu = \tfrac{1}{2}, \qquad
\mu \in \left(0, \tfrac{1}{2}\right).$$
\indent {\bf Case I: $\mu \in \left( \tfrac{1}{2},1 \right)$}. \\
We note that in this case $0_{\mu} \notin \Sigma_{\mu}$, and therefore
there is no effect of the singularity.
For every fixed $\lambda\in (1/2,\mu)$, by~\eqref{SysUV} we see that
\begin{equation*} %\label{eq.systemsolvedbyVlambda}
\left\{ \begin{array}{rl}
(\mathcal{L}_m +A(x)) \Vla = 0, & \textrm{in } \Sl,\\[0.15cm]
\Vla \geq 0, & \text{on $\partial \Sl$}, \\[0.15cm]
\Vla \not\equiv 0, & \textrm{on } \partial \Sl.
\end{array}\right. ,
\end{equation*}
where $\mathcal{L}_m$ is as in Lemma~\ref{MaxSmall} and
$A(x)$ is given by
\begin{equation} \label{eq.defiAonceandforall}
 A(x) = \left( \begin{array}{ccccc}
                -\alpha_1 & 1 & 0 & \ldots & 0\\
								0 & -\alpha_2 & 1 & \ldots & 0\\
								\vdots & \vdots & \ddots & \ddots & \vdots\\
								0 & 0 & 0 & -\alpha_{m-1} & 1\\
								c(x) & 0 & 0 & \ldots & -\alpha_m
						\end{array}\right).
\end{equation}
Now, fix a positive constant $\delta>0$ and a compact set $K \subset \Sigma_{\mu}$, such that
$$|\Sigma_{\mu} \setminus K| < \tfrac{\delta}{2}.$$
By definition of $\mu$, and since $K \subset \Sigma_{\mu}$ is compact,
it holds that
$$V_{\mu} > 0, \quad \textrm{in } K.$$
We now claim that there exits $\epsilon_0>0$ so small that for every $\epsilon\in(0,\epsilon_0)$:
\begin{equation}\label{Meas}
|\Sigma_{\mu-\epsilon} \setminus K| < \delta,
\end{equation}
\noindent and 
\begin{equation}\label{VinK}
V_{\mu - \epsilon}>0, \quad  \textrm{in }K.
\end{equation}
While~\eqref{Meas} follows by continuity, the claim in~\eqref{VinK} can be proved as follows: let
$$M_i := \min_{x \in K} \left\{ v_{i}^{(\mu)}(x) \right\} >0, \quad i=1,\ldots,m.$$
By continuity of the $u_i$'s, if $\epsilon_0$ is sufficiently small we have
$$|u_i^{(\mu - \epsilon)}(x)-u_i^{(\mu)}(x)| \leq \tfrac{M_i}{2} \quad \textrm{for every } x \in K.$$
Hence, for $x \in K$,
\begin{equation*}
v_i^{(\mu-\epsilon)}(x)= u_i^{(\mu - \epsilon)}(x) - u_i^{(\mu)}(x)+u_i^{(\mu)}(x)-u_i(x) \geq \tfrac{M_i}{2}>0,
\end{equation*}
\noindent in $K$, which proves~\eqref{VinK}.
In the complement of $K$ in $\Sigma_{\mu-\epsilon}$, we can then 
apply Lemma~\ref{MaxSmall} and the strong maximum principle
(to the each scalar equation of \eqref{SysUV}) getting
$$V_{\mu}> 0, \quad \textrm{in } \Sigma_{\mu-\epsilon},$$
\noindent which contradicts the definition of $\mu$. This excludes the case $\mu \in \left(\tfrac{1}{2},1\right)$.\\

\indent {\bf Case II: $\mu = \tfrac{1}{2}$}. \\
In this case, we are to distinguish two possible sub-cases.
If the $0_{\mu}$ does not lie in $\overline{\Omega}$, then there is no singularity
involved and we can argue exactly as in Case I. 
We have then to rule out just the case $0_{\mu} \in \partial \Omega$. 
Let us consider a positive constant $\delta >0$ small enough.
We define the set $D_{\delta} \subset \Smu$ as
\begin{equation*}
D_{\delta}:= \left\{ y \in \Smu: \mathrm{dist}(y,\partial \Smu) \geq \delta \right\}.
\end{equation*} 
Again, since $D_{\delta} \subset \Smu$, $V_{\mu}>0$ in $D_{\delta}$.
We now define the set $A_{\delta} \subset \overline{\Omega}$ as
\begin{equation*}
A_{\delta}:= \left\{ z \in \overline{\Omega}: |z-0_{\mu}| = \tfrac{\delta}{2} \right\}.
\end{equation*}
We note that $U$ is strictly positive on the ball $B_{\tfrac{\delta}{2}}(0)$
that surrounds the singularity. Therefore, $V_{\mu}>0$ on $A_{\delta}$.
This implies that
$$V_{\mu} >0, \quad \textrm{in } D_{\delta}\cup A_{\delta}.$$
We also note that $D_{\delta}\cap A_{\delta} = \emptyset$.
Arguing as in the proof of~\eqref{VinK},
 we can show that there exists $\epsilon_0
 = \epsilon_0(\delta) >0$ such that, for any $\epsilon\in [0,\epsilon_0]$, 
 we have $0_{\mu-\epsilon}\in \Omega\cap B_{\delta/2}(0_\mu)$ and
\begin{equation}\label{VinD}
V_{\mu-\epsilon} > 0, \quad \textrm{in } D_\delta \cup A_{\delta},
\end{equation}
\noindent for every $\epsilon \in (0,\epsilon_0)$. Moreover,
$$V_{\mu-\epsilon} \geq 0, \quad \textrm{on } \partial \left( \Sigma_{\mu-\epsilon} \setminus \left(D_{\delta}\cup B_{\tfrac{\delta}{2}}(0_{\mu}) \right) \right) \setminus T_{\mu-\epsilon},$$
\noindent but $V_{\mu-\epsilon} \not\equiv 0$ there. Therefore, for $\delta$ and $\epsilon$ small enough,
we can apply Lemma~\ref{MaxSmall} and the strong maximum principle 
(to the each scalar equation of \eqref{SysUV}) getting
$$V_{\mu-\epsilon} > 0, \quad \textrm{in } \Sigma_{\mu-\epsilon} \setminus \left( D_{\delta}\cup B_{\tfrac{\delta}{2}}(0_{\mu}) \right),$$
\noindent that, combined with~\eqref{VinD}, gives
$$V_{\mu-\epsilon} >0, \quad \textrm{in } \Sigma_{\mu-\epsilon}\setminus B_{\tfrac{\delta}{2}}(0_{\mu}).$$
To complete this second case we have to show that,
 setting $G := \Omega\cap B_{\delta/2}(0_\mu)$, we have
 $V_{\mu-\epsilon} \geq 0$ on $G\setminus\{0_{\mu-\epsilon}\}$.
 To this end, it suffices to apply Lemma~\ref{PerG} 
 with $A(x)$ as in~\eqref{eq.defiAonceandforall}: indeed, $G$ is clearly contained
 in $B_\delta(0_\mu)$; moreover, $V_{\mu-\epsilon}\geq 0$ on $\partial G$ 
 and, by \eqref{eq.vilambdabd},
 $v^{(\mu-\epsilon)}_i$ is bounded below
 on $G$ for every $i = 1,\ldots,m$.
% $$\inf_G v^{(\mu-\epsilon)}_i 
% = \inf_G\big(u_i^{(\mu-\epsilon)}-u_i) \geq \inf_G (-u_i) > -\infty.$$
 
 \noindent This proves that even the case $\mu =\tfrac{1}{2}$ is not possible. \medskip

 \indent {\bf Case III: $\mu \in \left(0, \tfrac{1}{2}\right)$}. 

\noindent To prove that also this case is not possible, we argue essentially as in
 {Case II}.
 First of all, given any $\delta > 0$ such that
 $\mathrm{dist}(0_\mu,\partial\Smu) > \delta$, we define
 $$D_{\delta} := \{ y \in \Smu : \mathrm{dist}(y,\partial\Smu)\geq \delta\}, \qquad
 K_\delta := D_\delta\setminus B_{\frac{\delta}{2}}(0_\mu)\subseteq\Smu.$$
 Since $K_\delta$ is compact and, by (iv), the function $V_\mu$
 is continuous and strictly positive on $K_\delta\subseteq\Smu\setminus\{0_\mu\}$, 
 it is possible to find
 $m_0 > 0$ such that
 $$v_i^{(\mu)} \geq m_0 \quad \text{on $K_\delta$} \qquad\quad \big(\text{for every $i = 1,\ldots,m$}\big).$$
 {F}rom this, by arguing as in {Case II}, we infer the existence of a small
 $\epsilon_0  =
 \epsilon_0(\delta) > 0$ such that, for every $\epsilon\in [0,\epsilon_0]$,
 the point
 $0_{\mu-\epsilon}$ lies in $D_\delta$ and
 \begin{equation} \label{eq.touseVmuepositive}
  \text{$V_{\mu-\epsilon} > 0$ on $K_\delta$}.
  \end{equation}
  By combining~\eqref{eq.touseVmuepositive}
  with~\eqref{eq.Vlambdasulbordogenerale},
 we infer that 
 $$\text{$V_{\mu-\epsilon} \geq 0$\,\,on 
 $\partial\Big(\Sigma_{\mu-\epsilon}\setminus D_\delta\Big)$ but 
 $V_{\mu-\epsilon} \not\equiv 0$ on the same set,}$$
 for every $\epsilon\in [0,\epsilon_0]$. 
 As a consequence, since $V_{\mu-\epsilon}$ solves the system of PDEs
 $$\text{$(\mathcal{L}_m + A(x))\,V_{\mu-\epsilon} = 0$ 
 on $\Sigma_{\mu-\epsilon}\setminus\{0_{\mu-\epsilon}\}$},$$
 (where $\mathcal{L}_m$ and $A(x)$ are as in Case I), if $\delta$ 
 has been chosen in such a way that $|\Sigma_{\mu-\epsilon_0}\setminus D_\delta|$ is small, we can
 apply Lemma~\ref{MaxSmall} and the strong maximum principle
 (again to the each scalar equation of \eqref{SysUV}), obtaining
 \begin{equation} \label{eq.touseVmuepositiveBIS}
  \text{$V_{\mu-\epsilon} > 0$\,\,on $\Sigma_{\mu-\epsilon}\setminus D_\delta$} \qquad
  \text{for every $\epsilon\in [0,\epsilon_0]$}.
 \end{equation}
 Gathering together~\eqref{eq.touseVmuepositiveBIS} and~\eqref{eq.touseVmuepositive} we conclude that,
 for every $0\leq\epsilon\leq\epsilon_0$,
 $$V_{\mu-\epsilon} > 0 \quad 
 \text{on $\Sigma_{\mu-\epsilon}\setminus B_{\frac{\delta}{2}}(0_\mu)$}.$$
 We now turn to prove that $V_{\mu-\epsilon} \geq 0$ throughout 
 $B_{\delta/2}(0_\mu)$ (for every fixed $\epsilon\in [0,\epsilon_0]$); 
 this will give a contradiction for the definition of $\mu$
 and will prove that $\mu = 0$, as desired.
  To this end, we argue as in the previous case: indeed, it suffices to apply once again Lemma~\ref{PerG}, 
 with $A(x)$ as in~\eqref{eq.defiAonceandforall}, on $B_{\delta/2}(0_\mu)$. 
The proof of~\eqref{DM} is therefore complete.

\noindent This, in the light of~\eqref{MM} establishes the first claim
in Theorem~\ref{thm.main}.\end{proof}

\begin{proof}[Proof of Theorem~\ref{thm.main} - (2)] 
We now complete the proof of Theorem~\ref{thm.main}
 by showing the monotonicity of $u$ with respect
 to $x_1$ (if $x_1 > 0$).
 First of all, by~\eqref{DM}, for every fixed $\lambda\in (0,1)$ we have
 that $V_\lambda > 0$ on $\Sigma_\lambda$.  
 Hence, in particular, 
 $$\text{$v^{(\lambda)}_1 > 0$ on $\SlPu$\quad and\,\,
 $(-\Delta + \alpha_1) v^{(\lambda)}_1 \stackrel{\eqref{SysUV}}{=} v^{(\lambda)}_2 > 0$ on 
 $\Sigma_{\lambda}\setminus\{0_\lambda\}$}.$$
 On the other hand, since (by definition) $v^{(\lambda)}_1 \equiv 0$ on $T_\lambda$, we
 are entitled to apply the classical Hopf's Lemma 
 (see, e.g., \cite[Lemma 3.4]{GT}), which gives that
 $$0 < \frac{\partial v_1^{(\lambda)}}{\partial x_1}(x) =  
 -2\,\frac{\partial u}{\partial x_1}(x)
 \quad \text{for every $x\in T_\lambda$}.$$
 This ends the proof of Theorem~\ref{thm.main}.\end{proof}

\begin{appendix}

\section{A combinatorics observation}

This appendix considers the setting in~\eqref{8A} and~\eqref{8B},
and checks the following combi\-na\-to\-rics remark:

\begin{lemma}\label{8}
We have that~$\alpha_1,\dots,\alpha_m\ge0$ if and only if~$s_0(\alpha),\dots,s_m(\alpha)\ge0$
\end{lemma}

\begin{proof}
Obviously, if~$\alpha_1,\dots,\alpha_{m}\ge0$,
we have that~$s_0(\alpha),\dots,s_m(\alpha)\ge0$, due to~\eqref{8B}.

Now, we prove that 
\begin{equation}\label{COMB}
{\mbox{if~$s_0(\alpha),\dots,s_{m-1}(\alpha)\ge0$
then necessarily~$\alpha_1,\dots,\alpha_m\ge0$.}}\end{equation}
To this end, we first prove~\eqref{COMB} under the additional assumption that
\begin{equation}\label{COMB0}
{\mbox{$\alpha_i\ne0$ for all $i\in\{1,\dots,m\}$.}}\end{equation}
To prove~\eqref{COMB}
under the additional assumption in~\eqref{COMB0}, we argue by induction over~$m$.
If~$m=1$, we have that~$s_0(\alpha)=\alpha_1$, thus the claim is obvious. 

Hence, to perform the inductive step,
we suppose now that
\begin{equation}\label{INAm78A}
{\mbox{\eqref{COMB} holds true for~$m$
replaced by~$m-1$,}}\end{equation} and we establish~\eqref{COMB}
for~$m\ge2$.

Up to reordering coordinates, we may suppose that
\begin{equation}\label{PALFA}
\alpha_1\ge\dots\ge\alpha_m.\end{equation}
Therefore, if~$\alpha_m\ge0$ we are done, and we can accordingly
assume also that
\begin{equation}\label{78HA89A}
\alpha_m<0.\end{equation}
We also set~$\alpha':=(\alpha_1,\dots,\alpha_{m-1})\in\R^{m-1}$.
We recall~\eqref{8B} and observe that, for every~$k\in\{1,\dots,m-1\}$,
\begin{eqnarray*}
s_k(\alpha)&=&\sum_{1\le i_1<\dots<i_{m-k}\le m-1}\alpha_{i_1}\dots\alpha_{i_{m-k}}+
\sum_{1\le i_1<\dots<i_{m-1-k}\le m-1}\alpha_{i_1}\dots\alpha_{i_{m-1-k}}\alpha_m
\\&=&
s_{k-1}(\alpha')+\alpha_m\,s_k(\alpha'),
\end{eqnarray*}
with the notation that the sum over the void set of indexes equals to~$1$. 
Consequently, for every~$k\in\{1,\dots,m-1\}$,
\begin{equation}\label{QU3A}
s_{k-1}(\alpha')=s_k(\alpha)-\alpha_m\,s_k(\alpha').
\end{equation}
We claim that
\begin{equation}\label{QU3A2}
s_k(\alpha')\ge0, \qquad{\mbox{ for all }}
k\in\{0,\dots,m-1\}.
\end{equation}
To prove this, we argue by backward induction over~$k$. First of all, we know that~$s_{m-1}(\alpha')=1$.
Then, suppose that~$s_j(\alpha')\ge0$ for all~$j\in\{k,\dots,m-1\}$, for some~$k\in\{1,\dots,m-1\}$.
Hence, by~\eqref{78HA89A},
$$ \alpha_m\,s_k(\alpha')\le0$$
and therefore, by~\eqref{QU3A},
$$ s_{k-1}(\alpha')=s_k(\alpha)-\alpha_m\,s_k(\alpha')\ge s_k(\alpha)\ge0.$$
This completes the inductive step and it proves~\eqref{QU3A2}.

Then, by~\eqref{QU3A2}, we are in the position of using~\eqref{INAm78A}
and deduce that~$\alpha_1,\dots,\alpha_{m-1}\ge0$. Hence, by~\eqref{COMB0},
we have that~$\alpha_1,\dots,\alpha_{m-1}>0$.
But then
$$ 0\le s_0(\alpha)=\alpha_1\dots\alpha_m<0,$$
which is a contradiction. This shows that~\eqref{78HA89A} cannot hold, and so, by~\eqref{PALFA},
$$ \alpha_1\ge\dots\ge\alpha_m\ge0,$$
and thus~\eqref{COMB} is established, under assumption~\eqref{COMB0}.

To remove the additional assumption in~\eqref{COMB0}, let now assume that~$\alpha_{m-\ell+1}=\dots=\alpha_m=0$, with~$\alpha_i\ne0$ for all~$i\in\{1,\dots, m-\ell\}$, for some~$\ell\in\{0,\dots,m\}$.
Then, letting~$\alpha^*:=(\alpha_1,\dots,\alpha_{m-\ell})$,
we have that
\begin{eqnarray*}
\sum_{k=0}^m s_k(\alpha)\,t^k&=&
\det\left( \begin{matrix}\alpha_1+t & 0 &\dots &0\\
0& \alpha_2+t&\dots&0\\
\vdots& \vdots& \ddots&\vdots\\
0&0&\dots&\alpha_m+t
\end{matrix}
\right)\\&=&
t^\ell
\det\left( \begin{matrix}\alpha_1+t & 0 &\dots &0\\
0& \alpha_2+t&\dots&0\\
\vdots& \vdots& \ddots&\vdots\\
0&0&\dots&\alpha_{m-\ell}+t
\end{matrix}
\right) \\&=& t^\ell
\sum_{k=0}^{m-\ell} s_k(\alpha^*)\,t^k.\end{eqnarray*}
This gives that
$$ 0\le s_{k+\ell}(\alpha) = s_k(\alpha^*),$$
for all~$k\in\{0,\dots,m-\ell\}$.

Since we have already proved~\eqref{COMB} is established under assumption~\eqref{COMB0}, we deduce that~$
\alpha_1,\dots,\alpha_{m-\ell}\ge0$,
which completes the proof of~\eqref{COMB}.
\end{proof}

\end{appendix}

%%%_________________________________________________________%%%

%%---------------------BIBLIOGRAPHY-----------------------%%

\end{document}